\documentclass[12pt, reqno]{amsart}
\usepackage{amsmath, amsthm, amscd, amsfonts, amssymb, graphicx, color}
\usepackage[bookmarksnumbered, colorlinks, plainpages]{hyperref}
\input{mathrsfs.sty}

\hypersetup{colorlinks=true,linkcolor=red, anchorcolor=green,
citecolor=cyan, urlcolor=red, filecolor=magenta, pdftoolbar=true}

\newtheorem{theorem}{Theorem}[section]
\newtheorem{lemma}[theorem]{Lemma}

\theoremstyle{definition}

\theoremstyle{remark}
\newtheorem{remark}[theorem]{Remark}
\numberwithin{equation}{section}

\newcommand{\R}{\mbox{$\mathbb{R}$}}

\newcommand{\C}{\mbox{$\mathbb{C}$}}
\newcommand{\K}{\mbox{$\mathbb{K}$}}

\newcommand{\fii}{\varphi}

\newcommand{\ip}[2]{\left\langle#1|#2\right\rangle}
\newcommand{\sip}[2]{\left[#1|#2\right]}

\DeclareMathOperator{\dist}{dist}
\DeclareMathOperator{\lin}{lin}

\usepackage{color}
\begin{document}

\setcounter{page}{1}

\title[On non-linear mappings preserving  the semi-inner product]
{On non-linear mappings preserving  the semi-inner product}
\author[T. Kobos \MakeLowercase{and} P. W\'ojcik]
{Tomasz Kobos$^1$ \MakeLowercase{and} Pawe\l{} W\'ojcik$^2$}
\address{$^1$Faculty of Mathematics and Computer Science, Jagiellonian University, \L ojasiewicza 6, 30-348 Krak\'ow, Poland}
\email{tomasz.kobos@uj.edu.pl}
\address{$^2$Institute of Mathematics, Pedagogical University of Cracow, Podchor\c a\.zych~2, 30-084 Krak\'ow, Poland}
\email{pawel.wojcik@up.krakow.pl}

\begin{abstract}
We say that a smooth normed space $X$ has a \textit{property (SL)}, if every mapping $f:X \to X$ preserving the semi-inner product on $X$ is linear. It is well known that every Hilbert space has the property (SL) and the same is true for every finite-dimensional smooth normed space. In this paper, we establish several new results concerning the property (SL). We give a simple example of a smooth and strictly convex Banach space which is isomorphic to the space $\ell_p$, but without the property (SL). Moreover, we provide a characterization of the property (SL) in the class of reflexive smooth Banach spaces in terms of subspaces of quotient spaces. As a consequence, we prove that the space $\ell_p$ have the property (SL) for every $1 < p < \infty$. Finally, using a variant of the Gowers-Maurey space, we construct an infinite-dimensional uniformly smooth Banach space $X$ such that every smooth Banach space isomorphic to $X$ has the property (SL).
\end{abstract} \maketitle
\section{Introduction}
Let $(X, \|\!\cdot\!\|)$ be a normed space over $\K$, where $\K=\R$ or $\K=\C$. A \emph{semi-inner product} on $X$ (in the sense of Lumer-Giles) is a mapping
$\sip{\cdot}{\cdot}\colon X\times X\to \K$ satisfying the conditions:
\begin{itemize}
\item[(sip1)] $\sip{\alpha x+\beta
y}{z}=\alpha\sip{x}{z}+{\beta}\sip{y}{z},\ \ x,y,z\in X,\
\alpha,\beta\in\K$;
\item[(sip2)] $\sip{x}{\lambda
y}=\overline{\lambda}\sip{x}{y},\ \ x,y\in X,\ \lambda\in\K$;
\item[(sip3)] $\big|\sip{x}{y}\big|\leq\|x\|\!\cdot\!\|y\|,\ \ x,y\in X$;
\item[(sip4)] $\sip{x}{x}=\|x\|^2,\ \ x\in X$.
\end{itemize}
Thus, in contrast to an inner product, a semi-inner product is not necessarily symmetric/conjugate-symmetric. This notion was first introduced by Lumer in \cite{lumer} and then later redefined by Giles in \cite{giles}. The main motivation behind it, was to transfer certain Hilbert space arguments to the setting of general Banach spaces. In fact, for every normed space there exists a semi-inner product generating the norm on $X$. This notion was later explored by many authors, as there exist some significant connections between a semi-inner product and the geometry of the underlying space. We refer to \cite{dragomir-sip} for some thorough investigation of the semi-inner products in Banach spaces. We shall quickly review some-well known properties of semi-inner products, that will be important for our purpose.

In a general setting, a semi-inner product generating the norm is not necessarily unique, which again contrasts with the case of an inner-product. However, the uniqueness of the semi-inner product can be easily characterized geometrically. Let us recall that a normed space $X$ is \emph{smooth} if for each nonzero vector $x \in X$ there is a unique \emph{supporting functional} $\varphi_x$ at $x$, that is, a linear functional satisfying the conditions $||\varphi_x||=1$ and $\varphi_x(x)=||x||$. It turns out that the uniqueness of a semi-inner product on $X$ is equivalent to the smoothness of $X$. Moreover, in this case, the unique semi-inner product on $X$ can be expressed by the formula
\begin{align}\label{sip-pattern-in-smooth-sp}
\sip{x}{y}=\|y\|\!\cdot\!\varphi_y(x),\quad x,y\in X,
\end{align}
where $\varphi_0$ is understood to be the zero functional. Therefore, we can say, that in the case of a smooth Banach space $X$, the semi-inner product is a certain canonical mapping associated to $X$, through which, a lot of important properties of $X$ can be expressed. For this reason, in this paper we shall focus our attention only on smooth Banach spaces, as this makes our investigation much clearer.

If the norm of $X$ is generated by an inner product $\ip{\cdot}{\cdot}$, we have the standard relation of orthogonality: $x\bot y \Leftrightarrow \ip{x}{y}=0$. In the case of a smooth Banach space $X$, the {\it orthogonality} can be defined in a similar way: $x\bot y \Leftrightarrow \sip{y}{x}=0$ (the reverse of order is important here). Since we consider the smooth space $X$, this notion of orthogonality turns out to coincide with the well-known relation of orthogonality introduced by Birkhoff in \cite{birkhoff} (see also \cite{james}). In other words
\begin{equation}\label{b-ort-sip}
x\bot y\ \ \Leftrightarrow \ \ \forall_{\alpha\in\begin{footnotesize} \K
\end{footnotesize}}\ \|x\|\leq\|x+\alpha y\|.
\end{equation}
As we shall consider only the semi-inner products, from this point we will use the symbol $\perp$ only for the relation orthogonality described above. To avoid confusion with the inner-product case, it should be kept in mind that this relation of orthogonality is not symmetric (in general).
Moreover, from (\ref{b-ort-sip}) it follows directly that:
\begin{align}\label{property-ort-dist}
{\rm if}\ Y\subseteq X\ {\rm is\ subspace\ and}\ a\bot Y, {\rm then}\ {\rm dist}(a,Y)=\|a\| .   
\end{align}
For a set $A \subseteq X$, we define the {\it orthogonal complement} of $A$ as 
$$A^{\perp} = \{ x \in X \ : \ x \perp a \text{ for every } a \in A\}.$$
If $(X, \|\!\cdot\!\|)$ is a reflexive Banach space, then for a
closed linear subspace $Y \subseteq X$, we have the {\it orthogonal decomposition} (see \cite[p. 163, 164]{dragomir-sip}):
\begin{align}\label{theorem-b-decomposition}
X=Y + Y^{\bot}
\end{align}
Moreover, if $X$ is additionally a stricly convex space, then every vector $x \in X$ has a unique representation of the form $x=y+z$, where $y \in Y$ and $z \in Y^{\bot}$. 

Here, we should strongly emphasis, that if $Y \subseteq X$ is a linear subspace, then generally $Y^{\perp}$ is not a linear subspace of $X$. Among many of interesting connections between the semi-inner product and some other notions of the Banach space theory, one relation of a particular importance is between the linearity of $Y^{\perp}$ and linear projections of norm one. If $Y$ is a linear subspace of a normed space $X$, then we say that $Y$ is \emph{$1$-complemented} in $X$, if there exists a linear projection $P: X \to Y$ of norm $1$. Projections of norm $1$ can be considered as a natural generalization of the orthogonal projections in an inner-product space to the setting of a general normed space. There is an extensive body of research devoted to studying $1$-complemented subspaces of normed spaces, both in general and specific settings. Linearity of the orthogonal complement is closely related to $1$-complemented subspaces. To illustrate this, let us consider a smooth and reflexive normed space $X$ and its subspace $Y$. If $Y^{\perp}$ is linear subspace, then representation of any vector $x \in X$ in a form $x=y+z$, where $y \in Y$, $z \in Y^{\perp}$ is unique. This representation allows us to define a linear projection $P:X \to X$ by $P(y+z)=z.$ Then we have $P(X)=Y^{\perp}$, $\ker P = Y$ and from the orthogonality it follows immediately that $\|P\|=1$. In particular,  $Y^{\perp}$ is $1$-complemented in $X$. Therefore, if $Y^{\perp}$ is a linear subspace, then it is automatically $1$-complemented in $X$. However, it is well-known that $1$-complemented subspaces are actually very rare, with the exception of one-dimensional subspaces, which are always $1$-complemented by the Hahn-Banach Theorem. From a result of Bosznay and Garay (see \cite{bosznay}) it follows, that if we consider the set of $n$-dimensional normed spaces endowed with some natural topology, then the majority of $n$-dimensional normed spaces does not have any $1$-complemented subspaces $Y$ of dimension $2 \leq \dim Y \leq n-1$. Moreover, if $X$ is a finite-dimensional normed space with $\dim X\geq 3$, then $Y^{\perp}$ is a linear subspace for all linear subspaces $Y\subseteq X$ if and only if, all linear subspaces $Y \subseteq X$ are $1$-complemented, if and only if $X$ is an inner-product space. For a proof of this famous result see for example \cite{philips}.  

We move to the main topic of the paper, that is to mappings preserving the semi-inner product. One of the fundamental properties of an inner product is the fact, that if $X$ is an inner product space and $f\colon X \to X$ is a mapping that preserves the inner product, then $f$ has to be a linear isometry. This classical result has a very short proof. Indeed, if we assume that $\ip{f(x)}{f(y)}=\ip{x}{y}$ for all $x, y \in X$, then for fixed vectors $x, y \in X$ and scalars $a, b \in \mathbb{K}$ we can consider the expression
$$||f \left ( ax + by \right ) - af(x) - bf(y) ||^2.$$

By using the linearity/conjugate-linearity of the inner product, we can rewrite this expression in the simplest terms of the form $\ip{f( \cdot)} {f( \cdot)}$. By using the fact that $f$ preserves the inner product, we can get rid of all $f$'s, and then again use the properties of inner product to obtain $0$. This shows that $f$ has to be linear. It is now immediate that $f$ has to be an isometry, as $f$ preserves also the norm of the vector; namely, we get $||f(x)||^2 = \ip{f(x)}{f(x)} = \ip{x}{x} = ||x||^2$.

The main objective of the paper is to study the analogous question for the semi-inner product. We say that a smooth normed space $X$ has \textit{property (SL)} if every mapping ${f\colon X\to X}$ satisfying the condition 
$$\sip{f(x)}{f(y)}=\sip{x}{y} \text { for } x, y \in X,$$ 
is linear. By the same reasoning as above, every mapping that preserves the semi-inner product preserves also the norm of the vector. Therefore, any linear mapping preserving the semi-inner product is automatically a linear isometry.

As we have already explained, every inner-product space has the property (SL). However, it is quite visible that proof of this fact is heavily based on the biadditivity of the inner-product. For this reason, the same argument can not be expected to work in the semi-inner product case. Ili\v{s}evi\'{c} and Turn\v{s}ek in \cite{ilisevic} have made an attempt to prove that every smooth normed space has the property (SL) (see \cite[Proposition 2.4]{ilisevic}. Unfortunately, their proof cointains a mistake, that was already pointed out in \cite{wojcik}. Perhaps surprisingly, it turns out that not only the proof of \cite[Proposition 2.4]{ilisevic} was wrong, but actually also the statement is not true. In the same paper \cite{wojcik} some smooth renorming of the space $\ell_p$ (with $1 < p < \infty$) was established, for which the resulting space does not have the property (SL). In other words, there exist an equivalent smooth norm $\|\!\cdot]\!\|$ on $\ell_p$ and a non-linear map $f\colon \big(\ell_p,\|\!\cdot \!\|\big) \to \big(\ell_p,\|\!\cdot \!\|\big)$, which preserves the semi-inner product. It should be noted, that proof of Ili\v{s}evi\'{c} and Turn\v{s}ek remains correct in the case, where $f\colon X \to X$ preserves the semi-inner product and is additionally a surjective mapping. Moreover, it was established in \cite{wojcik} that every smooth finite-dimensional normed space has the property (SL).

Already these first results demonstrate that property (SL) is highly connected to the geometry of the underlying space, but this relation is far from being obvious. The main goal of our paper is to explore this connection, by establishing several new results concerning the property (SL). First of all, we provide new simple and explicit construction of a smooth renorming of the space $\ell_p$, for which there is a non-linear mapping preserving the semi-inner product. Moreover, this equivalent norm is additionally also strictly convex, while in the previous construction (given in \cite{wojcik}) the lack of the strict convexity was crucial. Our result goes as follows.

\begin{theorem}
\label{twkonstrukcja}
Let $p\in (1,\infty)$. Let $X$ be any three-dimensional smooth and strictly convex normed space, which is not an inner-product space. Then, there exists a one-dimensional subspace $Y \subseteq X$, such that the space
$$\Big(W\oplus_p W \oplus_p\ldots\Big)\oplus_p\Big(X\oplus_p X\oplus_p\ldots\Big)$$
does not have the property (SL), where $W:= X/Y$ is the quotient space.
\end{theorem}

A natural question arises: does the space $\ell_p$, endowed with the standard norm $\|\!\cdot\!\|$, has the property (SL)? In order to answer the question, we give a characterization of the property (SL) in terms of subspaces of quotient spaces. Our results works for all smooth and reflexive Banach spaces and gives an equivalent condition to property (SL), but without referring to the semi-inner product on $X$. If $Y \subseteq X$ is a closed subspace, then by $k\colon X \to X/Y $ we denote the canonical surjection $k(\cdot):=[\cdot]$. By a proper closed
subspace $Y \subseteq X$ we mean a closed subspace such that $\{0\}\neq Y \neq X$.
\begin{theorem}
\label{twchar}
Let $X$ be a reflexive smooth Banach space. Then, $X$ has the property (SL) if and only if the following condition is satisfied: for every proper closed subspace $Y \subseteq X$ and for every a closed subspace $V\subseteq X/Y$ such that there is a linear surjective isometry $T\colon X \to V$, the set $k^{-1}(V) \cap Y^{\perp}$ is a linear subspace of $X$.
\end{theorem}
It is well-known that if $Y \subseteq \ell_p$ is a closed linear subspace such that $X/Y$ is isometric to $\ell_p$, then $Y^{\perp}$ is a linear subspace (see \cite{pelczynski} for the equivalent dual result). However, to prove that the space $\ell_p$ has the (SL) property, we need to take into the account not only $X/Y$ but all the possible subspaces of $X/Y$ that could be isometric to $\ell_p$.
It turns out that $\ell_p$ indeed has the (SL) property.
\begin{theorem}
\label{twlp}
For every $1 < p < \infty$ the space $\ell_p$ has the property (SL).
\end{theorem}
In other words, if $f: \ell_p \to \ell_p$ preserves the semi-inner product, then $f$ is a linear isometry. At the same time, by Theorem \ref{twkonstrukcja} there is an equivalent norm
$\|\!\cdot\!\|_o$ on $\ell_p$ such that some non-linear mapping $f\colon \big(\ell_p,\|\!\cdot\!\|_o\big)\to \big(\ell_p,\|\!\cdot\!\|_o\big)$ preserving the semi-inner product can be found. Therefore, one could ask, if there exists an inifnite-dimensional Banach space with the property (SL) for any possible renorming. Using Theorem \ref{twchar} we prove that the dual of a variant of the Gowers-Maurey space, introduced by Ferenczi in \cite{ferencziconvex}, has such a property.

\begin{theorem}
\label{twgm}
There exists an infinite-dimensional uniformly smooth Banach space $X$ such that every smooth Banach space isomorphic to $X$ has the property (SL).
\end{theorem}

The paper is organized as follows. In Section \ref{sectionquotient} we establish some auxiliary results, in which we develop certain relations between the semi-inner product on $X$ and the quotient spaces of $X$. In Section \ref{sectionkonstrukcja} we prove Theorem \ref{twkonstrukcja}, that is, we give an example of an infinite $\ell_p$ sum of Banach spaces, without the property (SL). In Section \ref{sectionchar} we prove Theorem \ref{twchar}, which characterizes the property (SL). Its consequences are presented in Section \ref{sectionconsequences}. Moreover, it is worth mentioning that Theorem \ref{twchar} has an interesting connection to an old open problem of the geometry of Banach spaces, proposed by Faulkner and Huneycutt in 1978 (\cite{faulkner}). We refer to Remark \ref{remarkproblem} for the details.

\section{Semi-inner product and the quotient space}
\label{sectionquotient}
\setcounter{equation}{0}
In order to construct a smooth and strictly convex Banach space without the property (SL), we need some auxiliary results. First, we find a formula for the semi-inner product on a quotient space. For a given closed linear subspace $Y \subseteq X$ the canonical surjection $k:X \to X/Y$ will be denoted by $k(x):=[x]$.
\begin{lemma}
\label{sipquotient}
Let $X$ be a reflexive smooth Banach space, and let $Y\subseteq X$ be a proper closed subspace. Then, the unique semi-inner product in the space $X/Y$ is given be the following formula
\begin{align}\label{sip-in-quotient-sp}
 \big[ [u]\, |\, [w]\big]=\sip{u_2}{w_2},
\end{align}
where
$$u=u_1 + u_2, \quad w = w_1 + w_2$$
are any decompositions such that $u_1, w_1 \in Y$ and $u_2, w_2 \in Y^{\perp}$.
\end{lemma}
\begin{proof}
Since $X$ is smooth, the quotient space $X/Y$ is also smooth and the semi-inner product in $X/Y$ is unique. Let $u, w \in X$ be arbitrary vectors and let $u=u_1 + u_2$, $w=w_1+w_2$ be decompositions like in the statement of the lemma. From property (\ref{property-ort-dist}) it follows that
$$\big\|[w]\big\|=\dist(w, Y)=||w_2||.$$
Let us consider a supporting functional $\fii_{[w]}\in (X/Y)^*$.
Both equalities $[w]=[w_2]$ and $\big\| [w]\big\|=\|w_2\|$ yield
\begin{center}
$(\fii_{[w]}\circ k)(w_2)=\fii_{[w]}\big([w_2]\big)=\fii_{[w]}\big([w]\big)=\big\| [w]\big\|=\|w_2\|$.
\end{center}
Moreover
$$||\fii_{[w]}\circ k || \leq ||\fii_{[w]}|| \cdot ||k|| \leq 1.$$
It follows that $\fii_{[w]}\circ k\in X^*$ is a supporting functional for $w_2$, so $\fii_{[w]}\circ k=\fii_{w_2}$.
Thus, we obtain
\begin{align*}
\big[ [u]\, |\, [w]\big]&\stackrel{\eqref{sip-pattern-in-smooth-sp}}{=}\big\| [w]\big\|\cdot \fii_{[w]}\big([u] \big)=\|w_2\|\cdot (\fii_{[w]}\circ k)(u)\\
&=\|w_2\|\cdot \fii_{w_2}(u)\stackrel{\eqref{sip-pattern-in-smooth-sp}}{=}\sip{u}{w_2}=\sip{u_1+u_2}{w_2}\\
&=\sip{u_1}{w_2}+\sip{u_2}{w_2}=0+\sip{u_2}{w_2}
\end{align*}
and the proof is complete.
\end{proof}
The next lemma is fundamental for our constructions of non-linear mappings preserving the semi-inner product.
\begin{lemma}
\label{lemquotient}
Let $X$ be a reflexive smooth Banach space. Let $Y \subseteq X$ be a proper closed subspace and let $V \subseteq X/Y$ be a closed subspace such that the set $k^{-1}(V) \cap Y^{\perp}$ is not a linear subspace of $X$. Then, there exists a non-linear mapping $f\colon V \to k^{-1}(V) \cap Y^{\perp}$ satisfying
\begin{align}\label{map-k-1}
\ \big[ f([u])\, |\, f([w])\big]=\sip{u}{w} \: \text{ for all } \: u, v \in V.
\end{align}
\end{lemma}
\begin{proof}
Let us denote $S := k^{-1}(V) \cap Y^{\perp}$. It follows from the property \eqref{theorem-b-decomposition} that the restriction $k|_S\colon S \to V$ is surjective. Let us first suppose that this restriction is also injective. Then the mapping $k|_S\colon S\to V$ is bijective and we can consider its inverse  $f\colon V\to S$. By our assumption, the set $S$ is not a
linear subspace, and therefore, the mapping $f$ is not linear. Moreover, if follows directly from Lemma~\ref{sipquotient} that $f$ preserves the semi-inner product for any two vectors from $V$.

Now, we assume that $k|_S\colon S \to V$ is not injective. In this case,
there are two vectors $z_1, z_2 \in  S$ such that $z_1\neq z_2$ and $[z_1]=[z_2] \in V$.

For all $x,y\in X$ we have $[x]\!=\![y]$ or $[x]\cap[y]\!=\!\emptyset$. Therefore, by the axiom of choice, there exists a subset $S'\!\subseteq\! S$ such that for each $[v] \in V \setminus \{[z_1]\}$ there exists exactly one $z \in S$ such that $[z]=[v]$.

Hence, a restriction $k|_{S'}\colon S'\to V\setminus \big\{[z_1]\big\}$ is a bijection. Now, we define two mappings $f_1,f_2\colon V\to S$  as:
\begin{center}
$f_1\big([x]\big):=\left\{\begin{array}{ccl}\big(k|_{S'}\big)^{-1}\big([x]\big)& {\rm if}& [x]\neq [z_1]=[z_2],\\
z_1& {\rm if}& [x]= [z_1]=[z_2]\end{array} \right.$
\end{center}
and similarly
\begin{center}
$f_2\big([x]\big):=\left\{\begin{array}{ccl}\big(k|_{S'}\big)^{-1}\big([x]\big)& {\rm if}& [x]\neq [z_1]=[z_2],\\
z_2& {\rm if}& [x]=[z_1]=[z_2].\end{array} \right. $
\end{center}
From the definition of $f_1$ and $f_2$, it follows that $f_1(v)=f_2(v)$ for all $v \in V$ such that $[v] \neq [z_1]$ and moreover $f_1(v) \neq f_2(v)$ for $[v]=[z_1]=[z_2]$. Two linear mappings can not differ at exactly one point. So, at least one of $f_1, f_2$ is not linear. Furthermore, again from Lemma \ref{sipquotient}, we can conclude that both of these mappings preserve the semi-inner product on $V$. This finishes the proof.
\end{proof}

\section{Sum of Banach spaces without the property (SL)}
\label{sectionkonstrukcja}

Let $\big(X_i, \|\! \cdot\!\|_i\big)_{i=1}^{\infty}$ be a sequence of Banach spaces and let $1 < p < \infty$. The $\ell_p$-sum of $X_i$ is denoted by
$$X_1 \oplus_p X_2 \oplus_p X_3 \oplus_p \ldots$$
and it is defined as a normed space consisting of sequences $(x_i)_{i=1}^{\infty}$ such that $x_i \in X_i$ for every $i \geq 1$ and $\big(||x_i||_i\big)_{i=1}^{\infty} \in \ell_p$. Norm is defined as 
$$\big\|(x_i)_{i=1}^{\infty}\big\|:=\left ( \sum_{i=1}^{\infty} ||x_i||^p_i \right )^\frac{1}{p}.$$
In the next lemma, we express the semi-inner product on the $\ell_p$-sum of Banach spaces by the semi-inner products of the individual spaces.
\begin{lemma}
\label{sipsuma}
Let $\big(X_i, \|\! \cdot\!\|_i\big)_{i=1}^{\infty}$ be a sequence of smooth Banach spaces and let $1 < p < \infty$. Suppose that $\sip{\cdot}{\cdot}_i$ is the semi-inner product in $X_i$ for $i \geq 1$. Then, the $\ell_p$-sum of $X_i$
$$X:=X_1 \oplus_p X_2 \oplus_p \ldots$$
is a smooth Banach space with the semi-inner product given by
$$\sip{x}{y} = \frac{1}{||y||^{p-2}} \left ( \sum_{i=1}^{\infty} ||y_i||^{p-2}_i \cdot \sip{x_i}{y_i}_i \right ),$$ 
where $x=(x_i)_{i=1}^{\infty}$ and $y=(y_i)_{i=1}^{\infty}$ are elements of $X_1 \oplus_p X_2 \oplus_p X_3 \oplus_p \ldots$.
\end{lemma}

\begin{proof}
Completeness of $X$ is well-known, while smoothness follows from a classical paper \cite{zachariades}. Hence, the semi-inner product on $X$ is unique and it is enough to verify that the given mapping satisfies all of the conditions (sip1-4). Properties (sip1), (sip2) and (sip4) follow easily from the fact that $\sip{\cdot}{\cdot}_i$ is the semi-inner product on $X_i$. To establish the condition (sip3), it is enough to use the H\"older's inequality. So, it follows that
\begin{align*}
\big|\sip{x}{y}\big| &= \frac{1}{||y||^{p-2}} \left| \sum_{i=1}^{\infty} ||y_i||^{p-2}_i \cdot \sip{x_i}{y_i}_i \right| \leq \frac{1}{||y||^{p-2}} \sum_{i=1}^{\infty} ||y_i||^{p-2}_i \cdot \big|\sip{x_i}{y_i}_i \big|\\
&\leq\frac{1}{||y||^{p-2}} \sum_{i=1}^{\infty} ||y_i||^{p-2}_i \cdot ||x_i||_i||y_i||_i =\frac{1}{||y||^{p-2}} \sum_{i=1}^{\infty} ||y_i||^{p-1}_i \cdot ||x_i||_i\\
&\leq \frac{1}{||y||^{p-2}} \left ( \sum_{i=1}^{\infty} ||y_i||^p_i \right )^{\frac{p-1}{p}} \left ( \sum_{i=1}^{\infty} ||x_i||^p_i \right )^{\frac{1}{p}}=||x|| \cdot ||y||,
\end{align*}
and we are done.
\end{proof}

Now, we  are in position to prove our first main result.

\emph{Proof of Theorem \ref{twkonstrukcja}}.
Let $X$ be any three-dimensional smooth and strictly convex normed space, which is not a Hilbert space. Since $X$ is not an inner-product space, there exits a one-dimensional subspace $Y \subseteq X$ such that $Y^{\bot}$ is not a linear subspace (otherwise, there would be a projection of norm $1$ on every subspace of $X$). We denote by $W:=X/Y$ the quotient space. It follows from Lemma \ref{lemquotient}, applied for $V=W$, that there is a non-linear function $f\colon W\to X$ such that
\begin{align}\label{sip-mapping-f-w-x}
\sip{f(u)}{f(w)}_{_X}=\sip{u}{w}_{_W}
\end{align}
for all $u, w \in W$.
Now, we define
\begin{center}
$Z:=\Big(W\oplus_p W\oplus_p W\oplus_p\ldots\Big)\oplus_p\Big(X\oplus_p X\oplus_p X\oplus_p\ldots\Big)$,
\end{center}
to be the $\ell_p$-sum as in the statement of Theorem \ref{twkonstrukcja}. The norm of $Z$ is thus given by
\begin{center}
$\|z\|:=\left(\sum\limits_{k=1}^\infty \|w_k\|_{_W}^p+\sum\limits_{j=1}^\infty \|x_j\|_{_X}^p\right)^{\frac{1}{p}}$
\end{center}
where $z=\big((w_1,w_2,\ldots),(x_1,x_2,\ldots)\big)\in Z$. It follows again from \cite{zachariades}, that $Z$ is a uniformly smooth and uniformly convex Banach space. So, by the previous Lemma \ref{sipsuma}, the unique semi-inner product $\sip{\cdot}{\cdot}\colon Z\times Z\to\K$ is given by the formula
\begin{align}\label{sip-in-space-z}
\sip{z}{a}:=\frac{1}{\|a\|^{p-2}}\left(\sum\limits_{k=1}^\infty \|b_k\|^{p-2}_{_W}\cdot\sip{w_k}{b_k}_{_W}+\sum\limits_{j=1}^\infty \|c_j\|^{p-2}_{_X}\cdot\sip{x_j}{c_j}_{_X}\right),
\end{align}
where
\begin{align*}
z&=\big((w_1,w_2,\ldots),(x_1,x_2,\ldots)\big)\in Z\ \  {\rm and}\\
a&=\big((b_1,b_2,\ldots),(c_1,c_2,\ldots)\big)\in Z.
\end{align*}
We are ready to construct a non-linear function $h\colon Z\to Z$ preserving the semi-inner product $\sip{\cdot}{\cdot}$. So, let us consider a mapping
\begin{align}\label{mapping-in-space-z}
h\big((w_1,w_2,\ldots),(x_1,x_2,\ldots)\big):=\big((w_2,w_3,\ldots),(f(w_1),x_1,x_2,\ldots)\big).
\end{align}
By combining \eqref{sip-in-space-z} with \eqref{sip-mapping-f-w-x} and \eqref{mapping-in-space-z}, we see that  $h$ preserves the semi-inner product on $Z$. Furthermore, since $f$ is not linear, $h$ is not linear as well and the the proof is complete.\qed

\begin{remark}
\label{remarklp}
It is quite easy to see that the renorming of $\ell_p$, like in Theorem \ref{twkonstrukcja}, can be chosen to be arbitrarily close to the standard norm. Indeed, let $1 < p < \infty$ be fixed. By a result of Bosznay and Garay (see \cite{bosznay}) it is possible to find a strictly convex and smooth norm $|| \cdot ||$ in $\mathbb{R}^3$ such that
$$||x|| \leq ||x||_p \leq (1 + \varepsilon)||x|| \: \text{ for every } x \in \mathbb{R}^n$$
and for every one-dimensional subspace $Y \subseteq X$, the set $Y^{\perp}$ is not a linear subspace. It can be easily verified, that for $Y = \{ (t, 0, 0) \ : t \in \mathbb{K}\ \}$, the norm in the quotient space $X/Y$ can be arbitrarily close to the two-dimensional $\ell_p$-norm (when $\varepsilon \to 0$). This means, that the  space $\ell_p$ has a smooth and strictly convex renorming, that is arbitrarily close to the standard norm; and for this renorming the property (SL) does not hold.
\end{remark}

\section{Characterization of the property (SL) and its applications}
\label{sectionchar}
The aim of this section is to prove Theorem \ref{twchar}, which characterizes the property (SL) in the class of reflexive smooth Banach spaces. Next, we will present some applications of it. Before proving Theorem \ref{twchar}, we need a simple lemma concerning linear isometries.

\begin{lemma}\label{lem-isom-t-pres-sip-x}
Let $X, Y$ be smooth Banach spaces and let $T\colon X \to Y$ be a linear surjective isometry. Then, for all $x, y \in X$ the equality
$$\sip{T(x)}{T(y)}_{Y}= \sip{x}{y}_X$$
is satisfied.
\end{lemma}
\begin{proof}
Since $X, Y$ are smooth, the semi-inner products in these spaces are uniquely determined. For a fixed nonzero $y \in X$, let us denote $g:= \varphi_y \circ T^{-1} \in Y^*$. Then
$$g(T(y)) = \varphi_y(y)=\|y\|=\|T(y)\|$$
and
$$\|g\| \leq \|\varphi_y\| \cdot \|T^{-1}\| = 1.$$
Thus, $g\in Y^*$ is the unique supporting functional for $T(y)$. Therefore, $g=\varphi_{T(y)}$.

Moreover, from (\ref{sip-pattern-in-smooth-sp}) we know that $\sip{x}{y}_X = ||y|| \varphi_y(x)$ and
\begin{align*}
\sip{T(x)}{T(y)}_Y &=\|T(y)\| \varphi_{T(y)}\big(T(x)\big) =\|y\| \varphi_{T(y)}\big(T(x)\big)=\|y\| g\big(T(x)\big)\\
&=\|y\| (\varphi_y \circ T^{-1})\big(T(x)\big)=\|y\|\varphi_y(x)=\sip{x}{y}_X
\end{align*}
and the proof is finished.
\end{proof}

\emph{Proof of Theorem \ref{twchar}}. We need to prove that two following conditions are equivalent:
\begin{enumerate}
\item There exists a non-linear mapping $f:X \to X$ that preserves the semi-inner product on $X$.
\item There exists a proper closed subspace $Y \subseteq X$ and a closed linear subspace $V \subseteq X/Y$ of the quotient space such that the set $k^{-1}(V) \cap Y^{\bot}$ is not a linear subspace and there exists a linear surjective isometry $T:X \to V$.
\end{enumerate}

We start with the implication $(2) \Rightarrow (1)$. From Lemma \ref{lemquotient} we know that there exists a non-linear mapping  $g: V \to X$ satisfying
\begin{align}
\ \sip{ g\big([u]\big)}{g\big([w]\big)}_X=\sip{u}{w}_V \: \text{ for all } \: u, v \in V.
\end{align}

Now, we define $f\colon X \to X$ as
$$f(x):= g\big(T(x)\big).$$
Since $T$ is linear and surjective, it is clear that $f$ is not linear. Moreover, by Lemma \ref{lem-isom-t-pres-sip-x}, mapping $T$ satisfies
$$\sip{T(x)}{T(y)}_V = \sip{x}{y}_X,$$
for every $x, y \in X$. Thus
$$\sip{f(x)}{f(y)}_X = \sip{g\big(T(x)\big)}{g\big(T(y)\big)}_X\stackrel{(4.1)}{=}\sip{T(x)}{T(y)}_V=\sip{x}{y}_X.$$
This concludes the proof of the implication $(2) \Rightarrow (1)$.

Now, let us suppose that the condition (1) is satisfied. We define $Y$ as
$$Y = \{ y \in X : \sip{y}{f(w)} = 0 \text{ for every } w \in X \}.$$
It follows directly from the properties of the semi-inner product, that $Y$ is a closed linear subspace of $X$. Let us fix two arbitrary scalars $\alpha, \beta \in \mathbb{K}$ and two arbitrary vectors $x, y \in X$. We shall prove that
\begin{align}
\label{kombinacja}
f(\alpha x + \beta y) - \alpha f(x) - \beta f(y) \in Y.
\end{align}
Indeed, if $w\in X$, then we have
\begin{align*}
&\sip{f(\alpha x + \beta y) - \alpha f(x) - \beta f(y)}{f(w)}=\\
&=
\sip{f(\alpha x + \beta y)}{f(w)}-
\alpha\sip{f(x)}{f(w)}-
\beta\sip{f(y)}{f(w)}\\
&=
\sip{\alpha x + \beta y}{w}-
\alpha\sip{x}{w}-
\beta\sip{y}{w}=0
\end{align*}
and (\ref{kombinacja}) is proved.

It is now evident that the mapping
$$X \ni x \to [f(x)] \in X/Y$$
is linear.  We will check that it is an isometry. Because $f(X)\bot Y$, we have $\dist\big(f(x),Y\big)=\|f(x)\|$ for all $x\in X$ from (\ref{b-ort-sip}). Hence
$$\big\|[f(x)]\big\| = \dist(f(x), Y) =\|f(x)\| = \|x\|,$$
where the last equality follows from the fact that $f$ preserves the norm.

Now, let $V$ be a linear subspace of $X/Y$, spanned by vectors of the form $[f(x)]$, where $x \in X$. Then $T\colon X \to V$ defined as $T(x):=[f(x)]$ is a desired linear surjective isometry from $X$ to $V$. Clearly, $V$ is a closed subspace, as an image of $X$ under a surjective linear isometry.

To finish the proof of the implication (1)$\Rightarrow$(2), it remains to show that the set $k^{-1}(V) \cap Y^{\perp}$ is not a linear subspace.

Let us assume otherwise. From the definitions of $Y$ and $V$ we have that $f(x) \in k^{-1}(V) \cap Y^{\perp}$ for every $x \in X$. If this set would form a linear subspace, then we would have also that
$$f(\alpha x + \beta y) - \alpha f(x) - \beta f(y) \in k^{-1}(V) \cap Y^{\perp} \subseteq Y^{\perp},$$
for every $\alpha, \beta \in \mathbb{K}$ and $x, y \in X$. From (\ref{kombinacja}) we know that vector of this form is also an element of $Y$ and hence it must be equal to $0$. But this means that $f$ is a linear mapping, which contradicts the assumption. This contradiction finishes the proof of the second implication. \qed

\begin{remark}
\label{remarkproblem}
At least in the strictly convex case it seems to be essential, that in Theorem \ref{twchar} we refer to a subspace $V \subseteq X/Y$ of a quotient space, and not just the quotient space itself. Faulkner and Huneycutt posed the following problem in 1978 (see \cite{faulkner}): is it true that if $X$ is a smooth reflexive Banach space and $Y \subseteq X$ is a closed subspace isometric to $X$, then $Y^{\perp}$ is a linear subspace? If $Y \subseteq X$ is such that $X/Y$ isometric to $X$, then by taking duals we get that $Y_0 \subseteq X^{*}$ is isometric to $X^{*}$, where 
$$Y_0 = \{ \varphi \in X^* \ : \ \varphi|_Y \equiv 0 \}.$$
If $X$ is strictly convex, then $X^*$ is smooth. In this case, if the problem of Faulkner and Huneycutt has the positive answer, we get that $Y^{\perp}_0$ is a linear subspace of $X^*$. From here, it is now quite easy to establish that $Y^{\perp}$ is also a linear subspace. Thus, it is not reasonable to expect that some easy example of the space without the property (SL) can be found, where only the quotient space (and not its subspace) is used. It should be noted that the problem of Faulkner and Huneycutt has positive answer in some particular cases, but in the full generality it remains open until this day. 
\end{remark}

\section{Some consequences of Theorem \ref{twchar}}
\label{sectionconsequences}

In this section we present some applications of Theorem \ref{twchar}. We begin with proving Theorem \ref{twlp}, which says that the classical $\ell_p$ space have the property (SL) for $1 < p < \infty$.

\emph{Proof of Theorem \ref{twlp}}.

By the Theorem \ref{twchar},  it is enough to prove that, if $Y \subseteq \ell_p$ is a closed linear subspace and $V \subseteq \ell_p/Y$ is a closed linear subspace of the quotient space, such that $V$ is linearly isometric to $\ell_p$, then the set 
$$S:=k^{-1}(V) \cap Y^{\perp}$$
is a linear subspace of $\ell_p$. 

Indeed, let $[z_1], [z_2], \ldots \in V$ be the image of the canonical unit basis of $\ell_p$ under the isometrical isomorphism. By the reflexivity of $\ell_p$ and \eqref{theorem-b-decomposition} we can assume that $z_i \in Y^{\perp}$ for every $i \geq 1$. In this case, we have $z_i \in S$ for every $i \geq 1$. Moreover, since $z_i\bot Y$, we get
$$1=\big\|[z_i]\big\|=\dist(z_i, Y) = ||z_i||.$$
We shall prove that 
$$k^{-1}(V) \cap Y^{\perp} = \overline{\lin \{ z_i \ : \ i=1,2,3, \ldots \}},$$
which will give us the desired conclusion, i.e. $k^{-1}(V) \cap Y^{\perp}$ is a linear subspace of $\ell_p$.

It is easy to see that if $x \in S$, then $\alpha x \in S$ for every $\alpha \in \mathbb{K}$. Therefore, it is enough to prove the following claim for all disjoint and finite $A, B \subseteq \mathbb{N}$ and all scalars $a_i, b_i$:
\begin{align}
\label{implikacja}
\sum_{i \in A}a_i z_i \in Y^{\perp}, \: \: \sum_{i \in B} b_iz_i \in Y^{\perp}\quad \Rightarrow\quad \sum_{i \in A}a_i z_i + \sum_{i \in B} b_iz_i \in Y^{\perp} 
\end{align}
Indeed, if the claim above is true, then by the fact that $S$ is closed, we immediately get that $\overline{\lin \{ z_i \ : \ i \geq 1 \}} \subseteq S$. On the other hand, if $v \in k^{-1}(V) \cap Y^{\perp}$, then $v \in Y^{\perp}$ and $[v] \in V$. Since vectors of the form $[z_1],[z_2],\ldots$ form the basis of $V$, we can write
$$[v] = \sum_{i=1}^{\infty} a_i [z_i] =  \left [\sum_{i=1}^{\infty} a_iz_i \right ],$$
for some scalars $a_i$. From the strict convexity of the space $\ell_p$ it follows that the representation in the form $y+y'$ (where $y \in Y$ and $y' \in Y^{\perp}$) is unique. Since $v\in Y^{\bot}$ and $\sum_{i=1}^{\infty} a_iz_i\in Y^{\bot}$, we must have $v = \sum_{i=1}^{\infty} a_iz_i$ and therefore $k^{-1}(V) \cap Y^{\perp} \subseteq \overline{\lin \{z_i \ : \ i \geq 1 \}}$.

Thus, it is enough to prove the implication (\ref{implikacja}).
Let us define $x:= \sum_{i \in A}a_i z_i$ and $y:=\sum_{i \in B} b_iz_i$ for disjoint and finite $A, B \subseteq \mathbb{N}$ and scalars $a_i, b_i$. If $x, y \in Y^{\perp}$, then it follows that
\begin{align*}
\dist(x+y, Y)^p &+ \dist(x-y, Y)^p = ||[x+y]||^p + ||[x-y]||^p\\
&=\left | \left | \sum_{i \in A}a_i [z_i] + \sum_{i \in b}b_i [z_i]\right | \right |^p + \left | \left |\sum_{i \in A}a_i [z_i] - \sum_{i \in b}b_i [z_i]\right |\right |^p\\
&= 2 \left ( \sum_{i \in A} |a_i|^p + \sum_{i \in B} |b_i|^p \right )=2\big(||[x]||^p + ||[y]||^p\big)\\
=& 2\big(\dist(x, Y)^p + \dist(y, Y)^p\big) = 2 \big(||x||^p + ||y||^p\big),
\end{align*}
where we have used respectively: the norm in the quotient space, the representation of $x$ and $y$ and linearity of the canonical map, the fact that $[z_1],[z_2],\ldots$ form a basis of $V$ and the assumption that $x, y$ belong to $Y^{\perp}$. Thus, from the equalities above, it follows that
\begin{equation}
\label{rownosc}
||x+y-m_1||^p + ||x-y-m_2||^p = 2(||x||^p + ||y||^p),
\end{equation}
where $m_1$, $m_2$ are best approximations in $Y$ for $x+y$ and $x-y$ respectively. Now, we will refer to the following inequalities proved by Hanner (see \cite{hanner}): if $p \geq 2$ the inequality
\begin{equation}
    ||u+v||^p + ||u-v||^p \geq 2 ( ||u||^p + ||v||^p),
\end{equation}
is true for any $u, v \in \ell_p$ and if $p \leq 2$, then the reverse inequality holds.

Let us assume first that $p \geq 2$. In this case, using the fact that $0$ is the best approximation for $x$ and $y$ in $Y$ and the Hanner inequality, we obtain
$$2(||x||^p + ||y||^p) \leq 2\left ( \left | \left |x - \frac{m_1+m_2}{2} \right | \right |^p + \left | \left |y - \frac{m_1-m_2}{2}\right | \right |^p \right)$$
$$\leq ||x + y  - m_1||^p + ||x-y-m_2||^p = 2(||x|| + ||y||^p).$$
This shows that $\frac{m_1+m_2}{2}$ is the best approximation to $x$ and $\frac{m_1-m_2}{2}$ is the best approximation to $y$, but since $\ell_p$ is a  strictly convex space, the best approximation is unique and we have that $m_1=m_2=0$. This shows that $0$ is the best approximation also to $x+y$ and therefore $x+y \in Y^{\perp}$.

In the case $p<2$, we proceed similarly. More precisely, we have
\begin{align*}
2(||x||^p + ||y||^p)&= ||x+y-m_1||^p + ||x-y-m_2||^p \\
&\leq ||x+y||^p + ||x-y||^{p}\leq 2\big(\|x\|^p+\|y\|^p \big).
\end{align*}
So, again from the Hanner inequality, it follows that this inequality is in fact an equality and similarly like before we have $m_1=0$. This means that $x+y \in M^{\perp}$ and the proof is finished. \qed

Even if the space $\ell_p$ has the property (SL), by Theorem \ref{twkonstrukcja} there are some Banach spaces isomorphic to $\ell_p$, without the property (SL). Now, we are going to prove that there exists an infinite-dimensional uniformly smooth Banach space $X$, such that any smooth space isomorphic to $X$ has the property (SL).

Let us recall that infinite-dimensional Banach space $X$ is called \emph{hereditarily indecomposable} if no closed subspace of $X$ can be represented as the topological sum of two infinite dimensional closed subspaces. The existence of heredatarily indecomposable Banach space is very far from being obvious. In a seminal paper \cite{gowers} Gowers and Maurey constructed the first example of such Banach space, which allowed them to answer negatively a famous long-standing open problem in the theory of Banach spaces called \emph{unconditional basic sequence problem}. After this breakthrough in Banach space theory, the Gowers-Maurey space was studied and considered by several different authors. Proof of our next Theorem is a quick consequence of two results of Ferenczi (\cite{ferencziconvex} and \cite{ferencziquotient}).

\emph{Proof of Theorem \ref{twgm}.} In \cite{ferencziconvex} it was proved that there exists a uniformly convex heredatarily indecomposable Banach space -- let us denote it by $X_F$. Every uniformly convex space is reflexive, so $X_F$ is reflexive. Let $X:=X_F^*$ be the dual space. Then $X$ is uniformly smooth reflexive Banach space and we shall show that $X$ has the desired property. So, let $X_0$ be any smooth Banach space linearly isomorphic to $X$ and let $S\colon X_0 \to X$ be a linear isomorphism. Clearly, $X_0$ is a reflexive space. Suppose that there exists a non-linear mapping $f\colon X_0 \to X_0$ which preserves the semi-inner product on $X_0$. Then, by Theorem \ref{twchar}, there exists a proper closed subspace $Y_0 \subseteq X_0$ and a closed subspace $V_0 \subseteq X_0/Y_0$ of the quotient space such that the set $k^{-1}(V_0) \cap Y_0^{\bot}$ is not a linear subspace and $V_0$ is linearly isometric to $X_0$.

Denote $Y:=S(Y_0)$. Clearly, the linear isomorphism $S\colon X_0 \to X$ induces an isomorphism between the quotient spaces $X_0/Y_0$ and $X/Y$. Let $V \subseteq X/Y$ be an image of $V_0$ in this induced isomorphism. Since $X_0$ is isomorphic (even isometric) to $V_0$, it follows that $X$ is isomorphic to $V$. In other words, $X$ is isomorphic to some closed subspace of its quotient $X/Y$, where $Y$ is a proper closed subspace. Dually, this means that $X^*=X_F$ is isomorphic to the quotient of its proper subspace. However, by the result of Ferenczi (see Proposition 3 in \cite{ferencziquotient}) hereditarily indecomposable Banach space can not be isomorphic to the proper quotient of subspace of itself. This contradiction concludes the proof. \qedsymbol
\begin{remark}
We mentioned in the Introduction, that it had been proved in \cite{wojcik} that every finite-dimensional normed space has the property (SL). This fact can be deduced immediately from Theorem \ref{twchar}. Indeed, a finite dimensional normed space can not be isometric to the subspace of its proper quotient which has a strictly smaller dimension.
\end{remark}

\section{Acknowledgements}

The research of the first author was funded by the Priority Research Area SciMat under the program Excellence Initiative – Research University at the Jagiellonian University in Kraków. We are also grateful to Piotr Niemiec and Anna Pelczar-Barwacz for valuable remarks concerning the manuscript.

\bibliographystyle{amsplain}

\end{document}